\documentclass[twoside,a4paper]{amsart}
\usepackage{amssymb,latexsym}
\input xypic
  \xyoption{all}

\usepackage{amsmath, amsthm, amsfonts}

\usepackage{pifont,pxfonts,txfonts}
\usepackage{yfonts}
\usepackage{bbm}
\usepackage{calrsfs}
\usepackage{empheq}
\usepackage{color}
\usepackage[colorlinks,linkcolor=blue,anchorcolor=blue,citecolor=blue]{hyperref}
\usepackage{amsmath, amstext, amsbsy, amscd}
\usepackage[mathscr]{eucal}
\usepackage{times}
\usepackage{amsmath, amsthm, amsfonts,extarrows}

\newtheorem{theorem}{Theorem}[section]
\newtheorem{proposition}[theorem]{Proposition}
\newtheorem{lemma}[theorem]{Lemma}

\newtheorem{corollary}[theorem]{Corollary}

\newtheorem{conjecture}[theorem]{Conjecture}

   \newcommand{\ba}{\begin{eqnarray}}
   \newcommand{\na}{\end{eqnarray}}
   \newcommand{\ban}{\begin{eqnarray*}}
   \newcommand{\nan}{\end{eqnarray*}}

\newcommand{\bC}{{\mathbb C}}

\newcommand{\bE}{{\mathbb E}}

\newcommand{\bP}{{\mathbb P}}

\newcommand{\bZ}{{\mathbb Z}}

\newcommand{\cO}{{\mathcal O}}

\newcommand{\Eff}{\textrm{Eff}}
\newcommand{\Pic}{\textrm{Pic}}
\newcommand{\Coeff}{\textrm{Coeff}}

  \newcommand{\<}{\langle}
  \renewcommand{\>}{\rangle}

\newcommand{\suml}{\sum\limits}
\newcommand{\prodl}{\prod\limits}

\begin{document}

\title{On conjecture $\cO$ for projective complete intersections}

\author[Hua-Zhong Ke]{Hua-Zhong Ke}
\address{School of Mathematics, Sun Yat-sen University, Guangzhou 510275, P.R. China;}
\email{kehuazh@mail.sysu.edu.cn}
\thanks{This work is supported by grants of National Natural Science Foundation of China (11601534, 11771461 and 11521101).}
\maketitle

\begin{abstract}
We prove that Fano complete intersections in projective spaces satisfy Conjecture $\cO$ proposed by Galkin-Golyshev-Iritani.
\end{abstract}



\date{\today}

\tableofcontents

\section{Introduction}

Let $F$ be a Fano manifold, i.e. a compact, complex manifold with ample anti-canonical line bundle. We set $H(F):=H^{even}(F,\bC)$. The quantum product $\star_0$ on $H(F)$ is defined by
\ban
\<\phi_1\star_{0}\phi_2,\phi_3\>^F=\suml_{d\in\Eff(X)}\<\phi_1,\phi_2,\phi_3\>_{0,d}^F,
\nan
for any $\phi_1,\phi_2,\phi_3\in H(F)$. Here $\<\cdot,\cdot\>^F$ is the Poincar\'e pairing of $X$, $\Eff(F)\subset H_2(F,\bZ)/tor$ is the set of effective curve classes of $F$, and $\<\phi_1,\phi_2,\phi_3\>_{0,d}^F$ is a genus-zero Gromov-Witten invariant of $F$. Then $(H(F),\star_0)$ is a unital, commutative and associative algebra.

\if{

Note that $F$ is Fano, and therefore for any integer $D$, there are only finitely many effective classes $d$ such that $\int_dc_1(F)=D$, which implies that RHS of  \eqref{quantumproduct} is well-defined.

}\fi

Galkin-Golyshev-Iritani \cite{GGI} conjectured that the distribution of eigenvalues of the linear operator 
\ban
(c_1(F)\star_0):H(F)\rightarrow H(F)
\nan
has some intriguing properties. The precise statement is as follows.

\begin{conjecture}(Conjecture $\cO$)

Let $\rho$ be the Fano index of $F$, and let $T(F)$ be the spectral radius of $(c_1(F)\star_0)$. Then:
\begin{enumerate}
\item $T(F)$ is an eigenvalue of $(c_1(F)\star_0)$ with multiplicity one;
\item if $u$ is an eigenvalue of $(c_1(F)\star_0)$ with $|u|=T(F)$, then $\frac u{T(F)}$ is a $\rho$-th root of unity.
\end{enumerate}
\end{conjecture}

To the knowledge of the author,  Conjecture $\cO$ was checked for several classes of Fano manifolds, including homogeneous spaces \cite{CL} and odd symplectic Grassmannians \cite{LMS}. It is natural to study complete intersections inside these manifolds.

The most basic known examples are projective spaces \cite{GGI}. In this article, we consider Conjecture $\cO$ for Fano complete intersections, i.e. smooth complete intersections in projective spaces which are Fano. Our main result is the following.

\begin{proposition}\label{main}
Fano complete intersections satisfy Conjecture $\cO$.
\end{proposition}

\if{

An important corollary of our main result is the following.

\begin{corollary}
Fano complete intersections satisfy Gamma conjecture $I$.
\end{corollary}

}\fi

Note that Conjecture $\cO$ was verified for del Pezzo surfaces \cite{HKLY}. So in this article, we only study Conjecture $\cO$ for an $N$-dimensional smooth Fano complete intersection $X$ of degree $(d_1,\cdots,d_r)$ in $\bP^{N+r}$, with $N\geq3$, $r\geq1$, $d_1,\cdots,d_r\geq2$. 

By Givental's mirror formulae for small $J$-functions, one can check that the ambient part of $H(X)$ satisfies Conjecture $\cO$. To prove the full version, we need to compute some genus-zero Gromov-Witten invariants (GWI) of $X$ with primitive insertions. Let $\rho$ be the Fano index of $X$, and we have the following three cases: (i) $N$ is odd; (ii) $N$ is even with $\rho>1$; (iii) $N$ is even with $\rho=1$. In the case (i), the primitive part is zero, and in the case (ii), the relevant invariants are zero from some known vanishing properties for GWI with primitive insertions and the dimension axiom. Galkin-Iritani \cite{GI} also used this observation to prove Conjecture $\cO$ for Fano hypersurfaces in the cases (i) and (ii).

In the case (iii), the dimension constraint is not strong enough, and we determine these GWI with primitive insertions by going from genus zero to genus one and back. We first use the genus-one \emph{topological recursion relation} to express the relevant genus-zero GWI with primitive insertions in terms of some genus-zero and genus-one GWI with only ambient insertions. Furthermore, using Zinger's \emph{standard versus reduced} formula, we observe that the above mentioned genus-one GWI are combinations of some genus-zero GWI with only ambient insertions. So, the relevant GWI with primitive insertions can be determined by some genus-zero, one-pointed and two-pointed GWI with only ambient insertions, which in turn can be reduced to one-pointed invariants by Lee-Pandharipande's \emph{divisor relations}. Finally, from Givental's \emph{mirror formula} for one-pointed invariants, we apply some generating function techniques to find the exact values of these relevant GWI with primitive insertions.

The trick of going from genus zero to genus one and back was first used by X. Hu \cite{H} to determine the quantum cohomology of cubic hypersurfaces. We expect that this trick is useful in the verification of Conjecture $\cO$ for complete intersections in other ambient spaces.

Though we will not go into details, we point out that Conjecture $\cO$ underlies Gamma conjecture I, and projective spaces and del Pezzo surfaces are known to satisfy Gamma conjecture I \cite{GGI, HKLY}. As a direct application of Proposition \ref{main}, we have the following corollary from Theorem 8.3 in \cite{GI}.
\begin{corollary}
Fano complete intersections satisfy Gamma conjecture $I$.
\end{corollary}

\if{
To the knowledge of the author, besides del Pezzo surfaces and Fano complete intersections, Conjecture $\cO$ is also verified for a few other cases \cite{CL, LMS}. 
}\fi

An outline of this article is as follows. We prove Conjecture $\cO$ for $X$ in the cases (i) and (ii) in Section 2, and we deal with the case (iii) in Section 3. In Section 4, we consider a related conjecture proposed by Galkin.

\section{Proof of the cases (i) and (ii)}

\if{

For $N\geq3$, $r\geq1$, $d_1,\cdots,d_r\geq2$, let $X$ be an $N$-dimensional smooth Fano complete intersection of degree $(d_1,\cdots,d_r)$ in $\bP^{N+r}$. Let 
\ban
\rho:=N+r+1-d_1-\cdots-d_r.
\nan 
Then $1\leq\rho\leq N$, and the adjunction formula tells us that
\ban
c_1(X)&=&\rho H,
\nan
where $H$ is the restriction of the hyperplane class of $\bP^{N+r}$ to $X$. So the index of $X$ is $\rho$.

We have three cases:
\begin{itemize}
\item[(i)] $N$ is odd;
\item[(ii)] $N$ is even with $\rho>1$;
\item[(iii)] $N$ is even with $\rho=1$.
\end{itemize}
To prove Conjecture $\cO$ for $X$, we will deal with the cases (i) and (ii) in this section, and leave the case (iii) to the next section.

}\fi

We follow notations in the Introduction. Recall that $\rho$ is the Fano index of $X$ with
\ban
\rho=N+r+1-d_1-\cdots-d_r,
\nan 
and $1\leq\rho\leq N$ by our assumption. The adjunction formula tells us that
\ban
c_1(X)&=&\rho H,
\nan
where $H$ is the restriction of the hyperplane class of $\bP^{N+r}$ to $X$.

Let $H_{amb}(X)$ and $H_{prim}(X)$ be the ambient part and the primitive part of $H(X)$, repectively. Then 
\ba
H_{amb}(X)&=&\bigoplus_{i=0}^N\bC H^i,\nonumber\\
H_{prim}(X)&=&0,\textrm{ if }N\textrm{ is odd.}\label{noprim}
\na
We have a direct sum decomposition of vector spaces:
\ban
H(X)=H_{amb}(X)\oplus H_{prim}(X).
\nan
Moreover, $H_{amb}(X)$ is a subalgebra of $(H(X),\star_0)$ generated by $H$. If $\rho>1$, then from Corollary 9.3 in \cite{Gi}, the relation for $H$ in $H_{amb}(X)$ is 
\ba\label{R1}
H^{\star_0(N+1)}=d_1^{d_1}\cdots d_r^{d_r}H^{\star_0(N+1-\rho)},
\na
and if $\rho=1$, then from Corollary 10.9 in \cite{Gi}, the relation is 
\ba\label{R2}
(H+d_1!\cdots d_r!)^{\star_0(N+1)}=d_1^{d_1}\cdots d_r^{d_r}(H+d_1!\cdots d_r!)^{\star_0 N}.
\na
Now \eqref{R1} and \eqref{R2} imply the following lemma.
\begin{lemma}\label{spectrum}
If $1<\rho\leq N$, then the spectrum of $(c_1(X)\star_0)$ on $H(X)$ is
\ban
\{0\}\cup\{e^{\frac{2\pi k\sqrt{-1}}{\rho}}(d_1^{d_1}\cdots d_r^{d_r})^{\frac 1\rho}\rho\}_{k=0}^{\rho-1};
\nan
if $\rho=1$, then the spectrum is 
\ban
\{-d_1!\cdots d_r!,\quad d_1^{d_1}\cdots d_r^{d_r}-d_1!\cdots d_r!\}.
\nan
\end{lemma}

The following Lemma \ref{oneprimitive} tells us that $H_{prim}(X)$ is a module of $H_{amb}(X)$.

\begin{lemma}\label{oneprimitive}
For any $i\geq0$ and $\gamma\in H_{prim}(X)$, we have $\<c_1(X),H^i,\gamma\>_{0,d}^X=0$.
\end{lemma}
\begin{proof}
This is a special case of Lemma 1 in \cite{LP}.
\end{proof}

From Lemma \ref{spectrum}, $T(X)$ is an eigenvalue of $(c_1(X)\star_0)$ on $H(X)$, and for any eigenvalue $u$ of $(c_1(X)\star_0)$ with $|u|=T(X)$, $\frac u{T(X)}$ is indeed a $\rho$-th root of unity. Moreover, as an eigenvalue of $(c_1(X)\star_0)$ acting on $H_{amb}(X)$, the multiplicity of $T(X)$ is one. So to prove Conjecture $\cO$ for $X$, we only need to show that $T(X)$ is not an eigenvalue of $(c_1(X)\star_0)$ acting on $H_{prim}(X)$. From \eqref{noprim}, we only need to consider the case of $N$ being even.

Since $N$ is even, it follows that $\<\cdot,\cdot\>^X$ is a symmetric, non-degenerate, bilinear form on $H_{prim}(X)$. Let
\ban
N':=\dim_\bC H_{prim}(X),
\nan
and let $\{\xi_i\}_{i=1}^{N'}$ be an orthonormal basis of $H_{prim}(X)$. Then Lemma \ref{oneprimitive} implies the following Lemma \ref{c1timesprim}.
\begin{lemma}\label{c1timesprim}
For $i=1,\cdots,N'$, we have $c_1(X)\star_0\xi_i=\suml_{j=1}^{N'}\<c_1(X),\xi_i,\xi_j\>_{0,1}^X\xi_j$.
\end{lemma}

From Lemma \ref{c1timesprim} and the dimension axiom, we have the following.
\begin{lemma}
If $\rho>1$, then $c_1(X)\star_0\xi_i=0$ for $i=1,\cdots,N'$.
\end{lemma}

So we have verified Conjecture $\cO$ for $X$ in the cases (i) and (ii). The remaining case (iii), in which $N$ is even with $\rho=1$, will be proved in the next section.

\section{Proof of the case (iii)}

In this section, we follow notations in the last section, and we assume that $N\geq3$ is even with $\rho=1$. We will show that  $T(X)$ is not an eigenvalue of $(c_1(X)\star_0)$
acting on $H_{prim}(X)$ (see the paragraph after Lemma \ref{oneprimitive}).

\if{

Throughout this section, let $X=X_N(d_1,\cdots,d_r)$ be an even-dimensional Fano complete intersection with index one, i.e. $N\geq4$ is even and $N+r=d_1+\cdots+d_r$. Then
\ban
c_1(X)&=&H,
\nan
where $H$ is the restriction of the hyperplane class of $\bP^{N+r}$ to $X$. 

Let $H^\cdot_{amb}(X)$ and $H^\cdot_{prim}(X)$ be the ambient part and the primitive part of $H^\cdot(X)$, repectively. Then we have a direct sum decomposition of vector spaces:
\ban
H^\cdot(X)=H^\cdot_{amb}(X)\oplus H^\cdot_{prim}(X).
\nan
From Corollary 10.9 in \cite{Gi}, $H^\cdot_{amb}(X)$ is a subalgebra of $(H^\cdot(X),\star_0)$, which is generated by $H$ with the relation
\ban
(H+d_1!\cdots d_r!)^{\star_0(N+1)}=d_1^{d_1}\cdots d_r^{d_r}(H+d_1!\cdots d_r!)^{\star_0 N}.
\nan

\begin{lemma}
The spectrum of $(H\star_0)$ on $H^\cdot(X)$ is 
\ban
\{-d_1!\cdots d_r!,\quad d_1^{d_1}\cdots d_r^{d_r}-d_1!\cdots d_r!\}.
\nan
\end{lemma}

}\fi

Since $N$ is even, it follows that $\dim_\bC H(X)=\chi_{top}(X)$, and hence
\ba\label{dimofprim}
N'=\chi_{top}(X)-(N+1).
\na
For $\rho=1$, it is well-known that $N'>0$, and we have $c_1(X)=H$.

\if{
Let $\{\xi_i\}_{i=1}^{N'}$ be an orthonormal basis of $H^\cdot_{prim}(X)$. 
\begin{lemma}
For $i=1,\cdots,N'$, we have $H\star_0\xi_i=\<H,\xi_i,\xi_i\>_{0,1}^X\xi_i$.
\end{lemma}
\begin{proof}
From the dimension constraint, we have
\ban
H\star_0\xi_i=\<H,\xi_i,H^{\frac N2}\>_{0,1}^X\frac{H^{\frac N2}}{d_1\cdots d_r}+\suml_{j=1}^{N'}\<H,\xi_i,\xi_j\>_{0,1}^X\xi_j.
\nan
On RHS, the first term is zero from Corollary 1.3 in \cite{H}, and the second term is equal to $\<H,\xi_i,\xi_i\>_{0,1}^X\xi_i$ from Theorem 1 in \cite{H}.
\end{proof}
}\fi

\begin{lemma}
For $i,j=1,\cdots,N'$, we have 
\ban
H\star_0\xi_i=\<H,\xi_i,\xi_i\>_{0,1}^X\xi_i,\textrm{ and }\<H,\xi_i,\xi_i\>_{0,1}^X=\<H,\xi_{j},\xi_{j}\>_{0,1}^X.
\nan
\end{lemma}
\begin{proof}
This follows from Theorem 1 in \cite{H}.
\end{proof}

So on $H_{prim}(X)$, $(H\star_0)$ is simply a scaling transformation with scale factor $\lambda:=\<H,\xi_1,\xi_1\>_{0,1}^X$. To prove Conjecture $\cO$ for $X$, we only need to show that 
\ba\label{lambda}
\lambda=-d_1!\cdots d_r!.
\na

To compute $\lambda$, we use the genus-one \emph{topological recursion relation} (see e.g. formula (3) in \cite{Ge}) to find:
\ba\label{TRR}
\<\tau_1(H)\>_{1,1}^X&=&\frac{1}{d_1\cdots d_r}\suml_{i=0}^{N}\<H,H^i\>_{0,1}^X\<H^{N-i}\>_{1,0}^X+\suml_{i=1}^{N'}\<H,\xi_i\>_{0,1}^X\<\xi_i\>_{1,0}^X\\
&&\quad+\frac{1}{d_1\cdots d_r}\suml_{i=0}^{N}\<H,H^i\>_{0,0}^X\<H^{N-i}\>_{1,1}^X+\suml_{i=1}^{N'}\<H,\xi_i\>_{0,0}^X\<\xi_i\>_{1,1}^X\nonumber\\
&&\quad+\frac{1}{24d_1\cdots d_r}\suml_{i=0}^{N}\<H,H^i,H^{N-i}\>_{0,1}^X+\frac{1}{24}\suml_{i=1}^{N'}\<H,\xi_i,\xi_i\>_{0,1}^X.\nonumber
\na
We can use the dimension axiom and the divisor axiom to simplify \eqref{TRR}:
\ba\label{TRRs}
\<\tau_1(H)\>_{1,1}^X&=&\frac{1}{d_1\cdots d_r}\<H^{N-1}\>_{0,1}^X\<H\>_{1,0}^X+\frac{1}{24d_1\cdots d_r}\suml_{i=0}^{N}\<H^i,H^{N-i}\>_{0,1}^X+\frac{1}{24}N'\lambda.
\na

The genus-one invariants in \eqref{TRRs} can be expressed in terms of genus-zero invariants with only ambient insertions, as shown in the following Lemma \ref{m11} and Lemma \ref{m10}.
\begin{lemma}\label{m11}
\ban
\<\tau_1(H)\>_{1,1}^X&=&-\frac1{24}\suml_{p=0}^{N-2}\bigg(\<\tau_p\big(c_{N-2-p}(X)\big)\tau_1\big(H\big)\>_{0,1}^X+\<\tau_p\big(c_{N-2-p}(X)\cup H\big)\>_{0,1}^X\bigg).
\nan
\end{lemma}
\begin{proof}
We use Zinger's \emph{standard versus reduced} formula to derive our result, and we follow notations in Theorem 1A in \cite{Z} to briefly explain the computation. Firstly, the corresponding reduced genus-one invariant is zero, since there is no genus-one, degree-one stable map to $X$ without contracting a subcurve of arithmetic genus one. Secondly, we have $m=1$, since a genus-zero, degree-zero stable map to $X$ has at least three marked points. Thirdly, for $(m,J)=(1,\emptyset)$ and $(m,J)=(1,\{1\})$, the corresponding coefficients of relavant genus-zero Gromov-Witten invariants can be obtained from formula (2-9) and formula (2-8) in \cite{Z}, respectively, which are both $-\frac1{24}$. Finally, summing over $p$ as in formula (2-10) in \cite{Z} gives the required equality.
\end{proof}

\begin{lemma}\label{m10} 
\ban
\<H\>_{1,0}^X=-\frac1{24}\int_XH\cup c_{N-1}(X).
\nan
\end{lemma}
\begin{proof}
Note that the obstruction bundle of $\overline M_{1,1}(X,0)=X\times\overline M_{1,1}$ is $T_X\boxtimes\bE^\vee$ (see e.g. Section 2 in \cite{GP}). So we have
\ban
\<H\>_{1,0}^X&=&-\int_{\overline M_{1,1}}\lambda_1\int_XH\cup c_{N-1}(X)=-\frac1{24}\int_XH\cup c_{N-1}(X).
\nan
\end{proof}

So from \eqref{TRRs}, Lemma \ref{m11} and Lemma \ref{m10}, we have
\ba\label{twopoint}
N'\lambda&=&-\suml_{p=0}^{N-2}\bigg(\<\tau_p\Big(c_{N-2-p}(X)\Big)\tau_1(H)\>_{0,1}^X+\<\tau_p\Big(c_{N-2-p}(X)\cup H\Big)\>_{0,1}^X\bigg)\\
&&\quad+\frac{1}{d_1\cdots d_r}\<H^{N-1}\>_{0,1}^X\int_XH\cup c_{N-1}(X)-\frac{1}{d_1\cdots d_r}\suml_{i=0}^{N}\<H^i,H^{N-i}\>_{0,1}^X.\nonumber
\na

The RHS of \eqref{twopoint} can be expressed in terms of genus-zero one-point invariants. To this end, we need the following Lemma \ref{divisorrelation}.
\begin{lemma}\label{divisorrelation}
Let $Y$ be a nonsingular, projective, complex algebraic variety, $\beta$ an irreducible curve class of $Y$ and $L\in\Pic(Y)$. Then on $\overline M_{0,2}(Y,\beta)$, we have
\ban
ev_1^*(L)\cap[\overline M_{0,2}(Y,\beta)]^{vir}&=&\bigg(ev_2^*(L)+\int_\beta c_1(L)\psi_2\bigg)\cap[\overline M_{0,2}(Y,\beta)]^{vir},\\
\psi_1\cap[\overline M_{0,2}(Y,\beta)]^{vir}&=&-\psi_2\cap[\overline M_{0,2}(Y,\beta)]^{vir}.
\nan
\end{lemma}
\begin{proof}
This is a special case of Lee-Pandharipande's \emph{divisor relations} (see Corollary 1 in \cite{LP}). Here we do not have the terms coming from distributing marked points and degrees, since a degree-zero, genus-zero stable map to $X$ has at least three marked points.
\end{proof}

Now we use Lemma \ref{divisorrelation} to reduce two-point invariants on RHS of \eqref{twopoint} to one-point invariants.

\begin{lemma}\label{cherntwopoint}
\ban
\<\tau_p\Big(c_{N-2-p}(X)\Big)\tau_1(H)\>_{0,1}^X+\<\tau_p\Big(c_{N-2-p}(X)\cup H\Big)\>_{0,1}^X=-\<\tau_{p+1}\Big(c_{N-2-p}(X)\Big)\>_{0,1}^X.
\nan
\end{lemma}
\begin{proof}
We have
\ban
&&\<\tau_p\Big(c_{N-2-p}(X)\Big)\tau_1(H)\>_{0,1}^X+\<\tau_p\Big(c_{N-2-p}(X)\cup H\Big)\>_{0,1}^X\\
&=&-\<\tau_{p+1}\Big(c_{N-2-p}(X)\Big)\tau_0(H)\>_{0,1}^X+\<\tau_p\Big(c_{N-2-p}(X)\cup H\Big)\>_{0,1}^X\\
&=&-\<\tau_{p+1}\Big(c_{N-2-p}(X)\Big)\>_{0,1}^X-\<\tau_p\Big(c_{N-2-p}(X)\cup H\Big)\>_{0,1}^X+\<\tau_p\Big(c_{N-2-p}(X)\cup H\Big)\>_{0,1}^X\\
&=&-\<\tau_{p+1}\Big(c_{N-2-p}(X)\Big)\>_{0,1}^X.
\nan
Here we use Lemma \ref{divisorrelation} to derive the first equality, and we use the divisor axiom to derive the second equality. 
\end{proof}

\begin{lemma}\label{puretwopoint}
\ban
\<\tau_0(H^i)\tau_a(H^{N-i-a})\>_{0,1}^X=\suml_{p=0}^i{i\choose p}\<\tau_{a+p-1}(H^{N-a-p})\>_{0,1}^X.
\nan
\end{lemma}
\begin{proof}
For $i=0$, the required equality holds by the fundamental class axiom. For $i>0$, from Lemma \ref{divisorrelation}, we have
\ban
\<\tau_0(H^i)\tau_a(H^{N-i-a})\>_{0,1}^X=\<\tau_0(H^{i-1})\tau_a(H^{N-(i-1)-a})\>_{0,1}^X+\<\tau_0(H^{i-1})\tau_{a+1}(H^{N-(i-1)-(a+1)})\>_{0,1}^X.
\nan
Now we use induction on $i=0,1,2,3,\cdots$ to obtain the required result.
\end{proof}

So from \eqref{twopoint}, Lemma \ref{cherntwopoint} and Lemma \ref{puretwopoint}, we obtain
\ba\label{onepoint}
N'\lambda&=&\suml_{p=0}^{N-2}\<\tau_{p+1}\Big(c_{N-2-p}(X)\Big)\>_{0,1}^X+\frac{1}{d_1\cdots d_r}\<H^{N-1}\>_{0,1}^X\int_XH\cup c_{N-1}(X)\\
&&\quad-\frac{1}{d_1\cdots d_r}\suml_{i=0}^{N}\suml_{p=0}^i{i\choose p}\<\tau_{p-1}(H^{N-p})\>_{0,1}^X.\nonumber
\na
To manipulate RHS of \eqref{onepoint}, we introduce integers $c_p$'s and $I_p$'s as follows. We define $c_p$'s by
\ban
\frac{(1+x)^{N+r+1}}{\prodl_{i=1}^r(1+d_ix)}=\suml_{p=0}^\infty c_px^p.
\nan
Then
\ba
c(X)&=&\suml_{p=0}^{N}c_pH^p,\nonumber\\
\chi_{top}(X)&=&(d_1\cdots d_r)c_N.\label{topeuler}
\na
We define $I_p$'s by
\ban
(d_1\cdots d_r)(d_1!\cdots d_r!)\bigg[\frac{\prodl_{i=1}^r\prodl_{m=1}^{d_i}(1+\frac{d_i}{m}x)}{(1+x)^{N+r+1}}-1\bigg]=\suml_{a=0}^{\infty}I_ax^a.
\nan
Then the \emph{mirror formula} for $X$ (see Theorem 10.7 in \cite{Gi}) implies that $I_a=\<\tau_{a-1}(H^{N-a})\>_{0,1}^X$ for $0\leq a\leq N$. In particular, $I_0=0$. Using $c_p$'s and $I_p$'s, we see from \eqref{onepoint} that
\ban
N'\lambda&=&\suml_{p=0}^{N-2}c_{N-p-2}I_{p+2}+c_{N-1}I_1-\frac{1}{d_1\cdots d_r}\suml_{i=0}^{N}\suml_{p=0}^i{i\choose p}I_p\\
&=&\suml_{p=0}^{N}c_{N-p}I_{p}-\frac{1}{d_1\cdots d_r}\suml_{p=0}^{N}\suml_{i=p}^{N}{i\choose p}I_p\\
&=&\suml_{p=0}^{N}\bigg[c_{N-p}-\frac{1}{d_1\cdots d_r}{N+1\choose p+1}\bigg]I_p.
\nan
Here we use $I_0=0$ in the second equality. So 
\ban
N'\lambda=\Coeff_{x^{N}}\big(g(x)\big),
\nan
where
\ban
g(x)=\bigg[\frac{(1+x)^{N+r+1}}{\prodl_{i=1}^r(1+d_ix)}-\frac{(1+x)^{N+1}}{d_1\cdots d_r}\bigg]\cdot(d_1\cdots d_r)(d_1!\cdots d_r!)\bigg[\frac{\prodl_{i=1}^r\prodl_{m=1}^{d_i}(1+\frac{d_i}{m}x)}{(1+x)^{N+r+1}}-1\bigg].
\nan
Here for a rational function $R(x)$ holomorphic at $x=0$, we use its Taylor expansion at the origin $R(x)=\suml_{k=0}^\infty R_kx^k$ to set
\ban
\Coeff_{x^k}\big(R(x)\big):=R_k,\quad k\in\bZ_{\geq0}.
\nan
Now direct calculation gives
\ban
g(x)&=&(d_1\cdots d_r)(d_1!\cdots d_r!)\bigg[\prodl_{i=1}^r\prodl_{m=2}^{d_i}(1+\frac{d_i}m x)-\frac{(1+x)^{N+r+1}}{\prodl_{i=1}^r(1+d_ix)}\\
&&\quad-\frac{1}{d_1\cdots d_r}\prodl_{i=1}^r\prodl_{m=1}^{d_i-1}(1+\frac{d_i}m x)+\frac{(1+x)^{N+1}}{d_1\cdots d_r}\bigg].
\nan
Note that $N=d_1+\cdots+d_r-r$, and then we have
\ban
N'\lambda&=&\Coeff_{x^{N}}\big(g(x)\big)\\
&=&(d_1\cdots d_r)(d_1!\cdots d_r!)\bigg[\prodl_{i=1}^r\prodl_{m=2}^{d_i}\frac{d_i}m-\Coeff_{x^{N}}\bigg(\frac{(1+x)^{d_1+\cdots+d_r+1}}{\prodl_{i=1}^r(1+d_ix)}\bigg)\\
&&\quad-\frac{1}{d_1\cdots d_r}\prodl_{i=1}^r\prodl_{m=1}^{d_i-1}\frac{d_i}m+\frac1{d_1\cdots d_r}{N+1\choose N}\bigg]\\
&=&(d_1!\cdots d_r!)\bigg[(N+1)-(d_1\cdots d_r)\Coeff_{x^{N}}\bigg(\frac{(1+x)^{d_1+\cdots+d_r+1}}{\prodl_{i=1}^r(1+d_ix)}\bigg)\bigg]\\
&=&(d_1!\cdots d_r!)\bigg[(N+1)-\chi_{top}(X)\bigg]\\
&=&(d_1!\cdots d_r!)(-N').
\nan
We use \eqref{topeuler} in the fourth equality, and use \eqref{dimofprim} in the last equality. Since $N'>0$, it follows that $\lambda=-d_1!\cdots d_r!$, which verifies \eqref{lambda}. This finishes the proof of Conjecture $\cO$ in the case (iii).

\section{A conjecture of Galkin}

We follow notations in Section 2.

Galkin \cite{Ga} conjectured that, for a Fano manifold $F$, $T(F)\geq\dim_\bC F+1$, with equality only if $F$ is a projective space. This conjecture was verified for del Pezzo surfaces \cite{HKLY}. Together with $T(\bP^N)=N+1$, the following Lemma \ref{Galkinconjecture} verifies Galkin's conjecture for Fano complete intersections of dimension at least three. 

\begin{lemma}\label{Galkinconjecture}
$T(X)>N+1$. 
\end{lemma}
\begin{proof}
If $1<\rho\leq N$, then $T(X)=(d_1^{d_1}\cdots d_r^{d_r})^{\frac 1\rho}\rho$. From the convexity of the function $x\mapsto x\log x$ with $x>0$, we have
\ban
T(X)\geq(1+\frac{N+1-\rho}{r})^{\frac{r+N+1-\rho}{\rho}}\rho.
\nan
Note that $1\leq r\leq N+1-\rho$. Since the function $x\mapsto (1+\frac{N+1-\rho}{x})^{x+N+1-\rho}$ with $1\leq x\leq N+1-\rho$ is decreasing, it follows that
\ban
T(X)\geq4^{\frac{N+1-\rho}{\rho}}\rho.
\nan
The function $x\mapsto 4^{\frac{N+1-x}{x}}x$ with $2\leq x\leq N+1$ is strictly decreasing, and as a consequence, we have $T(X)>N+1$.

If $\rho=1$, then 
\ban
T(X)=d_1\cdots d_r\big(d_1^{d_1-1}\cdots d_r^{d_r-1}-(d_1-1)!\cdots(d_r-1)!\big)>d_1\cdots d_r.
\nan
For the case $r=1$, we have $d_1=N+1$, and the required inequality follows. For the case $r>1$, note that $xy>x+y-1$ for $x,y>1$, and therefore, 
\ban
T(X)&>&(d_1+d_2-1)d_3\cdots d_r\\
&>&(d_1+d_2+d_3-2)d_4\cdots d_r\\
&>&\cdots\\
&>&d_1+\cdots+d_r-(r-1)=N+1.
\nan
\end{proof}

\if{
\section{Lower bounds of the spectral radius}

Let $Q^N$ be the smooth quadric hypersurface in $\bP^{N+1}$.

\begin{proposition}
Let $X$ be a Fano complete intersection with dimension $N\geq3$ and index $\rho$, and let $T$ be the spectral radius of $(c_1(X)\star_0)$. Then:
\begin{enumerate}
\item if $\rho\leq N+1$, then $\frac T\rho\geq1$, with equality only if $\rho=N+1$, i.e. $X\cong\bP^N$;
\item if $\rho\leq N$, then $\frac T\rho\geq 4^{\frac 1N}$, with equality only if $\rho=N$, i.e. $X\cong Q^N$.
\end{enumerate}
\end{proposition}

}\fi

{\bf Acknowledgements.}
The author would like to thank Xiaowen Hu for enlightening discussions on Gromov-Witten invariants with primitive insertions, and Changzheng Li for helpful discussions on Conjecture $\cO$. The author is grateful to Jianxun Hu for constant encouragement and support.

\end{document}